\documentclass[a4paper]{amsart}
\usepackage{graphicx,amsmath,amsfonts,latexsym,amssymb,amsthm,mathrsfs, color,hyperref}
\usepackage[latin1]{inputenc}
\newtheorem{theorem}{Theorem}
\newtheorem*{remark*}{Remark}

\begin{document}
\title[Heinz-Kato inequality in Banach spaces]%
{Heinz-Kato inequality in Banach spaces}
\author{Nikolaos Roidos}
\address{Department of Mathematics, University of Patras, 26504 Rio Patras, Greece}
\email{roidos@math.upatras.gr}

\thanks{The author was supported by Deutsche Forschungsgemeinschaft, grant SCHR 319/9-1.}
\date{\today}
\begin{abstract} 
It is observed that in Banach spaces, sectorial operators having bounded imaginary powers satisfy a Heinz-Kato inequality.
\end{abstract}
\maketitle

The Heinz-Kato inequality, see \cite{He} and \cite{Ka2}, is formulated as follows:

\begin{theorem}[Heinz-Kato]\label{HeKa} Let $H_{1}$, $H_{2}$ be Hilbert spaces and $A$, $B$ be positive selfadjoint operators in $H_{1}$, $H_{2}$ respectively. Let $T$ be a bounded linear operator from $H_{1}$ to $H_{2}$ such that $T(\mathcal{D}(A))\subseteq \mathcal{D}(B)$ and 
$$
\|BTu\|_{H_{2}}\leq M\|Au\|_{H_{1}},\quad u\in\mathcal{D}(A),
$$
for certain $M\geq0$. Then, for each $a\in(0,1)$ we have that $T(\mathcal{D}(A^{a}))\subseteq \mathcal{D}(B^{a})$ and
$$
\|B^{a}Tu\|_{H_{2}}\leq M^{a}\|T\|_{\mathcal{L}(H_{1},H_{2})}^{1-a}\|A^{a}u\|_{H_{1}},\quad u\in\mathcal{D}(A^{a}).
$$
\end{theorem}

The above theorem was extended in \cite{Ka1} to maximal accretive operators in Hilbert spaces. Moreover, in \cite[Theorem 2.31]{Ya} it was extended to invertible sectorial operators in a Hilbert space that have bounded imaginary powers.

Let $X$ be a complex Banach space. A closed densely defined linear operators $A$ in $X$ is called {\em invertible sectorial} if
$$
[0,+\infty)\subset\rho{(-A)} \quad \mbox{and} \quad (1+s)\|(A+s)^{-1}\|_{\mathcal{L}(X)}\leq K, \quad s\in[0,+\infty),
$$
for certain $K\geq1$. In this situation a perturbation argument, see, e.g., \cite[(III.4.6.2)-(III.4.6.3)]{Am}, implies that 
$$
\Omega_{K}=\Big\{\lambda\in\mathbb{C}\, | \, |\arg(\lambda)|\leq\arcsin\Big(\frac{1}{2K}\Big)\Big\}\cup\Big\{\lambda\in\mathbb{C}\, | \, |\lambda|\leq\frac{1}{2K}\Big\}\subset\rho{(-A)}
$$
and
\begin{gather*}
(1+|\lambda|)\|(A+\lambda)^{-1}\|_{\mathcal{L}(E)}\leq 2K+1, \quad \forall \lambda\in \Omega_{K}.
\end{gather*}

Denote by $\partial\Omega_{K}$ the positively oriented boundary of $\Omega_{K}$. The complex powers $A^{z}$ of $A$ for $\mathrm{Re}(z)<0$ are defined by the Dunford integral formula
$$
A^{z}=\frac{1}{2\pi i}\int_{\partial\Omega_{K}}(-\lambda)^{z}(A+\lambda)^{-1}d\lambda,
$$
see, e.g., \cite[Theorem III.4.6.5]{Am}; in particular, if $a\in(0,1)$ then 
$$
A^{-a}=\frac{\sin(\pi a)}{\pi}\int_{0}^{\infty}s^{-a}(A+s)^{-1}ds.
$$
The family $\{A^z\}_{\mathrm{Re}(z)<0}\cup\{A^{0}=I\}$ is a strongly continuous holomorphic semigroup on $X$, see, e.g., \cite[Theorem III.4.6.2 ]{Am}. Moreover, each $A^{z}$, $\mathrm{Re}(z)<0$, is an injection and the complex powers for positive real part $A^{-z}$ are defined by $A^{-z}=(A^{z})^{-1}$. The imaginary powers $A^{it}$, $t\in\mathbb{R}$, are defined by the closure of the operator 
$$
\mathcal{D}(A)\ni u\mapsto \frac{\sin(i\pi t)}{i\pi t}\int_{0}^{\infty}s^{it}(A+s)^{-2}Auds \quad \text{in} \quad X,
$$
see, e.g., \cite[(III.4.6.21)]{Am}; they can be either bounded or unbounded operators in $X$.

If there exist $\delta,C>0$ such that $A^{it}\in\mathcal{L}(X)$ and $\|A^{it}\|_{\mathcal{L}(X)}\leq C$ when $t\in[-\delta,\delta]$, then we say that {\em $A$ has bounded imaginary powers}. In this situation, $A^{it}\in\mathcal{L}(X)$ for all $t\in\mathbb{R}$, $\{A^{it}\}_{t\in\mathbb{R}}$ forms a strongly continuous group and there exist some $\phi_{A}\geq0$ and $M_{A}\geq1$ such that $\|A^{it}\|_{\mathcal{L}(X)}\leq M_{A}e^{\phi_{A}|t|}$, $t\in\mathbb{R}$, see, e.g., \cite[Theorem III.4.7.1]{Am} and \cite[Corollary III.4.7.2]{Am}. Moreover, the family $\{A^{z}\}_{\mathrm{Re}(z)\leq0}$ is a strongly continuous semigroup on $\mathcal{L}(X)$, see, e.g., \cite[Theorem III.4.7.1]{Am}. For further properties of the complex powers of invertible sectorial operators we refer to \cite[Section III.4.6]{Am}, \cite[Section III.4.7]{Am} and \cite[Section 2]{Tan}. 

A closed densely defined injective linear operator $B$ in $X$ is called {\em sectorial} if
$$
(0,+\infty)\subset\rho{(-B)} \quad \mbox{and} \quad s\|(B+s)^{-1}\|_{\mathcal{L}(X)}\leq L, \quad s\in(0,+\infty),
$$
for certain $L\geq1$. If $B$ is sectorial, then similarly to \cite[(III.4.6.2)-(III.4.6.3)]{Am} we have that 
$$
S_{L}=\Big\{\lambda\in\mathbb{C}\, | \, |\arg(\lambda)|\leq\arcsin\Big(\frac{1}{2L}\Big)\Big\}\subset\rho{(-B)}
$$
and
\begin{gather*}
|\lambda|\|(B+\lambda)^{-1}\|_{\mathcal{L}(E)}\leq 2L-1, \quad \forall \lambda\in S_{L}.
\end{gather*}
For each $m,k\in\mathbb{N}$ and $\eta\in\mathbb{C}$ such that $|\mathrm{Re}(\eta)|<m$ denote
$$
Q_{A}(\eta,m,k)=\frac{1}{2\pi i}\int_{\partial S_{L}}\Big(\frac{k}{k-\lambda}-\frac{1}{1-k\lambda}\Big)^{m}(-\lambda)^{\eta}(B+\lambda)^{-1}d\lambda,
$$
where $\partial S_{L}$ is the positively oriented boundary of $S_{L}$. The $\eta$-complex power of $B$ is defined pointwise as
$$
B^{\eta}u=\lim_{k\rightarrow\infty}Q_{A}(\eta,m,k)u, \quad u\in \mathcal{D}(B^{\eta}),
$$
where 
$$
\mathcal{D}(B^{\eta})=\{u\in X\, |\, \text{$\lim_{k\rightarrow\infty}Q_{A}(\eta,m,k)u$ exists in $X$}\}.
$$
$B^{\eta}$ is a well defined closed linear operator in $X$ independent from $m$; it is in general unbounded and satisfies $\mathcal{D}(B^{m})\cap\mathcal{R}(B^{m})\subseteq \mathcal{D}(B^{\eta})$, where $\mathcal{R}(\cdot)$ denotes the range. In particular, $\mathcal{D}(B^{0})=X$ and $B^{0}=I$. Moreover, $\mathcal{D}(B^{\gamma})=\mathcal{D}((B+\mu)^{\gamma})$ for each $\gamma,\mu>0$, and the norms $\|\cdot\|_{X}+\|B^{\gamma}\cdot\|_{X}$, $\|(B+\mu)^{\gamma}\cdot\|_{X}$ are equivalent on $\mathcal{D}(B^{\gamma})$, see \cite[Lemma 15.22]{KW}. For further details on the complex powers of sectorial operators through the notion of the extended holomorphic functional calculus we refer to \cite[Section 15]{KW}.

By using the boundedness of the imaginary powers property we obtain the following generalization of Theorem \ref{HeKa}.

\begin{theorem}\label{Theorem2}
Let $X_{1}$, $X_{2}$ be Banach spaces and $A$, $B$ be sectorial operators in $X_{1}$, $X_{2}$ respectively that have bounded imaginary powers; i.e. $A^{it}\in \mathcal{L}(X_{1})$ and $B^{it}\in \mathcal{L}(X_{2})$ for each $t\in \mathbb{R}$ with 
$$
\|A^{it}\|_{\mathcal{L}(X_{1})}\leq M_{A}e^{\phi_{A}|t|}, \quad \|B^{it}\|_{\mathcal{L}(X_{2})}\leq M_{B}e^{\phi_{B}|t|}, \quad t\in\mathbb{R},
$$ 
for certain $M_{A},M_{B}\geq1$ and $\phi_{A},\phi_{B}\geq0$. Let $T$ be a bounded linear operator from $X_{1}$ to $X_{2}$ such that $T(\mathcal{D}(A))\subseteq \mathcal{D}(B)$ and 
$$
\|BTu\|_{X_{2}}\leq M\|Au\|_{X_{1}},\quad u\in\mathcal{D}(A),
$$
for certain $M\geq0$. Then, for each $a\in(0,1)$ we have that $T(\mathcal{D}(A^{a}))\subseteq \mathcal{D}(B^{a})$ and 
$$
\|B^{a}Tu\|_{X_{2}}\leq M_{A}M_{B}M^{a}e^{\frac{\phi_{A}^{2}+\phi_{B}^{2}}{4}+2(\max\{a,1-a\})^{2}}\|T\|_{\mathcal{L}(X_{1},X_{2})}^{1-a}\|A^{a}u\|_{X_{1}}, \quad u\in \mathcal{D}(A^{a}).
$$
\end{theorem}
\begin{proof}
Follows by \cite[Theorem 15.28]{KW} and interpolation. More precisely, for each $b\in[0,1]$ denote by $X_{A,b}$ and $X_{B,b}$ the completion of $(\mathcal{D}(A^{b}),\|A^{b}\cdot\|_{X_{1}})$ and $(\mathcal{D}(B^{b}),\|B^{b}\cdot\|_{X_{2}})$ respectively. Consider the map $\widetilde{T}:\{v_{k}\}_{k\in\mathbb{N}}\mapsto \{Tv_{k}\}_{k\in\mathbb{N}}$, where $\{v_{k}\}_{k\in\mathbb{N}}$ is a sequence in $X_{1}$. If $\{v_{k}\}_{k\in\mathbb{N}}$ is Cauchy in $X_{1}$ then $\{Tv_{k}\}_{k\in\mathbb{N}}$ is Cauchy in $X_{2}$ so that $\widetilde{T}$ can be extended to belong to $\mathcal{L}(X_{A,0},X_{B,0})$ with $\|\widetilde{T}\|_{\mathcal{L}(X_{A,0},X_{B,0})}=\|T\|_{\mathcal{L}(X_{1},X_{2})}$. Moreover, if $v_{k}\in \mathcal{D}(A)$, $k\in\mathbb{N}$, and $\{Av_{k}\}_{k\in\mathbb{N}}$ is Cauchy in $X_{1}$, then $Tv_{k}\in \mathcal{D}(B)$, $k\in\mathbb{N}$, and $\{BTv_{k}\}_{k\in\mathbb{N}}$ is Cauchy in $X_{2}$, so that $\widetilde{T}\in\mathcal{L}(X_{A,1},X_{B,1})$ and $\|\widetilde{T}\|_{\mathcal{L}(X_{A,1},X_{B,1})}\leq M$.

Due to \cite[Proposition 15.25]{KW} each of $(X_{A,0},X_{A,1})$, $(X_{B,0},X_{B,1})$ is an interpolation couple. If we denote by $[\cdot,\cdot]_{a}$ the complex interpolation functor of type $a$, then by \cite[Lemma 15.22]{KW} and \cite[Theorem 15.28]{KW} we have that $X_{A,a}=[X_{A,0},X_{A,1}]_{a}$ and 
$$
(M_{A}e^{\frac{\phi_{A}^{2}}{4}+(\max\{a,(1-a)\})^{2}})^{-1}\|\cdot\|_{[X_{A,0},X_{A,1}]_{a}}\leq \|\cdot\|_{X_{A,a}}\leq M_{A}e^{\frac{\phi_{A}^{2}}{4}+(\max\{a,1-a\})^{2}}\|\cdot\|_{[X_{A,0},X_{A,1}]_{a}},
$$
and similarly for $B$, where the bounds in the above inequality follow by the estimates in the proof of \cite[Theorem 15.28]{KW}. Moreover, by \cite[Theorem 2.6]{Lu}, i.e. since $[\cdot,\cdot]_{a}$ is exact of type $a$ (see, e.g., \cite[Theorem 1.9.3]{Tri}), we have that $\widetilde{T}\in\mathcal{L}([X_{A,0},X_{A,1}]_{a},[X_{B,0},X_{B,1}]_{a})$ and 
$$
\|\widetilde{T}\|_{\mathcal{L}([X_{A,0},X_{A,1}]_{a},[X_{B,0},X_{B,1}]_{a})}\leq \|\widetilde{T}\|_{\mathcal{L}(X_{A,0},X_{B,0})}^{1-a}\|\widetilde{T}\|_{\mathcal{L}(X_{A,1},X_{B,1})}^{a}.
$$
Hence, $\widetilde{T}\in\mathcal{L}(X_{A,a},X_{B,a})$ and 
$$
\|\widetilde{T}v\|_{X_{B,a}}\leq M_{A}M_{B}M^{a}e^{\frac{\phi_{A}^{2}+\phi_{B}^{2}}{4}+2(\max\{a,1-a\})^{2}}\|T\|_{\mathcal{L}(X_{1},X_{2})}^{1-a}\|v\|_{X_{A,a}}, \quad v\in X_{A,a}.
$$

If we let $v=\{v_{k}\}_{k\in\mathbb{N}}$ with $v_{k}=u\in \mathcal{D}(A^{a})$, $k\in\mathbb{N}$, then $v\in X_{A,a} $ and $\|u\|_{X_{A,a}}=\|A^{a}u\|_{X_{1}}$. Moreover, $\widetilde{T}u=\{Tv_{k}\}_{k\in\mathbb{N}}$ with $Tv_{k}=Tu\in X_{2}$, $k\in\mathbb{N}$, so that $\widetilde{T}u\in X_{B,0}\cap X_{B,a}$ and hence $Tu\in \mathcal{D}(B^{a})$ due to \cite[Proposition 15.26 (d)]{KW}, which implies that $\|\widetilde{T}u\|_{X_{B,a}}=\|B^{a}Tu\|_{X_{2}}$.
\end{proof}

If the sectorial operators are invertible, then by following the ideas in the proofs of \cite[Theorem 2.3.3]{Tan} and \cite[Theorem 2.31]{Ya} we give an alternative proof of Theorem \ref{Theorem2} as follows.

\begin{theorem} Let $X_{1}$, $X_{2}$ be Banach spaces and $A$, $B$ be invertible sectorial operators in $X_{1}$, $X_{2}$ respectively that have bounded imaginary powers; i.e. they satisfy 
$$
\|A^{it}\|_{\mathcal{L}(X_{1})}\leq M_{A}e^{\phi_{A}|t|} \quad \text{and} \quad \|B^{it}\|_{\mathcal{L}(X_{2})}\leq M_{B}e^{\phi_{B}|t|}, \quad t\in\mathbb{R},
$$ 
for certain $M_{A},M_{B}\geq1$ and $\phi_{A},\phi_{B}\geq0$. Let $T$ be a bounded linear operator from $X_{1}$ to $X_{2}$ such that $T(\mathcal{D}(A))\subseteq \mathcal{D}(B)$ and 
$$
\|BTu\|_{X_{2}}\leq M\|Au\|_{X_{1}},\quad u\in\mathcal{D}(A),
$$
for certain $M\geq0$. Then, for each $a\in(0,1)$ we have that $T(\mathcal{D}(A^{a}))\subseteq \mathcal{D}(B^{a})$ and
$$
\|B^{a}Tu\|_{X_{2}}\leq M_{A}M_{B} M^{a}e^{(\phi_{A}+\phi_{B})\sqrt{a(1-a)}}\|T\|_{\mathcal{L}(X_{1},X_{2})}^{1-a}\|A^{a}u\|_{X_{1}},\quad u\in\mathcal{D}(A^{a}).
$$
\end{theorem}
\begin{proof}
Denote by $X_{2}^{\ast}$ the dual space of $X_{2}$ and let $v\in X_{2}^{\ast}$, $u\in\mathcal{D}(A^{1+a})$. By \cite[Theorem III.4.6.5]{Am}, \cite[Theorem III.4.7.1]{Am} and \cite[(III.4.7.3)]{Am} the map 
$$
z\mapsto f(z)= e^{\xi z(z-1)}v(B^{-z}BTA^{-1}A^{z-1}A^{1+a}u)
$$ 
is analytic in $\{z\in\mathbb{C}\, |\, 0<\mathrm{Re}(z)<1\}$ and continuous on $\{z\in\mathbb{C}\, |\, 0\leq\mathrm{Re}(z)\leq 1\}$, where $\xi=\frac{\phi}{2\sqrt{a(1-a)}}$, $\phi=\phi_{A}+\phi_{B}$. Moreover, for each $b\in [0,1]$ and $t\in\mathbb{R}$ we have that
\begin{eqnarray*}
\lefteqn{|f(b+it)|=e^{\xi(b(b-1)-t^{2})}|v(B^{-b}B^{-it}BTA^{-1}A^{b-1}A^{it}A^{1+a}u)|}\\
&\leq& M_{A}M_{B}e^{\xi b(b-1)}e^{\phi|t|-\xi t^{2}}\|v\|_{X_{2}^{\ast}}\|B^{-b}\|_{\mathcal{L}(X_{2})}\|BTA^{-1}\|_{\mathcal{L}(X_{1},X_{2})}\|A^{b-1}\|_{\mathcal{L}(X_{1})}\|A^{1+a}u\|_{X_{1}},
\end{eqnarray*}
so that $f$ is bounded on $\{z\in\mathbb{C}\, |\, 0\leq\mathrm{Re}(z)\leq 1\}$ due to \cite[Theorem III.4.6.4]{Am}.

We estimate
\begin{eqnarray*}
\lefteqn{|f(it)|=e^{-\xi t^{2}}|v(B^{-it}BTA^{-1}A^{it}A^{a}u)|}\\
&\leq&M_{A}M_{B}Me^{\phi|t|-\xi t^{2}}\|v\|_{X_{2}^{\ast}}\|A^{a}u\|_{X_{1}}\\
&\leq&e^{\frac{\phi\sqrt{a(1-a)}}{2}}M_{A}M_{B}M\|v\|_{X_{2}^{\ast}}\|A^{a}u\|_{X_{1}}.
\end{eqnarray*}
Similarly,
\begin{eqnarray*}
\lefteqn{|f(1+it)|=e^{-\xi t^{2}}|v(B^{-it}TA^{it}A^{a}u)|}\\
&\leq&M_{A}M_{B}e^{\phi|t|-\xi t^{2}}\|T\|_{\mathcal{L}(X_{1},X_{2})}\|v\|_{X_{2}^{\ast}}\|A^{a}u\|_{X_{1}}\\
&\leq&e^{\frac{\phi\sqrt{a(1-a)}}{2}}M_{A}M_{B}\|T\|_{\mathcal{L}(X_{1},X_{2})}\|v\|_{X_{2}^{\ast}}\|A^{a}u\|_{X_{1}}.
\end{eqnarray*}
Therefore, by the Hadamard's three lines theorem (see, e.g., \cite[Lemma 1.1.2]{BeLo}) we obtain
\begin{eqnarray*}
\lefteqn{|f(1-a)|=e^{-\frac{\phi\sqrt{a(1-a)}}{2}}|v(B^{a}TA^{-a}A^{a}u)|}\\
&\leq&\Big(\sup_{t\in\mathbb{R}}|f(it)|\Big)^{a}\Big(\sup_{t\in\mathbb{R}}|f(1+it)|\Big)^{1-a}\\
&\leq&e^{\frac{\phi\sqrt{a(1-a)}}{2}}M_{A}M_{B}M^{a}\|T\|_{\mathcal{L}(X_{1},X_{2})}^{1-a}\|v\|_{X_{2}^{\ast}}\|A^{a}u\|_{X_{1}}.
\end{eqnarray*}
We conclude that 
\begin{eqnarray*}
\|B^{a}TA^{-a}w\|_{X_{2}}\leq e^{\phi\sqrt{a(1-a)}}M_{A}M_{B}M^{a}\|T\|_{\mathcal{L}(X_{1},X_{2})}^{1-a} \|w\|_{X_{1}}, \quad w\in \mathcal{D}(A),
\end{eqnarray*}
so that by the closedness of $B^{a}$, $T$ maps $\mathcal{D}(A^{a})$ to $\mathcal{D}(B^{a})$ and the result follows.
\end{proof}

\begin{remark*}
Concerning the boundedness of the imaginary powers, for examples of operators satisfying this property we refer, e.g., to \cite{AHS} for elliptic differential operators with smooth coefficients, to \cite{DS} for elliptic differential operators with nonsmooth coefficients and to \cite{SS} for degenerate differential operators. 
\end{remark*}

\end{document}